\documentclass[12pt]{amsart}
  
 \usepackage{amsfonts,graphics,amsmath,amsthm,amsfonts,amscd,amssymb,amsmath,latexsym,multicol, 
 mathrsfs}
\usepackage{epsfig,url}
\usepackage{flafter}
\usepackage{fancyhdr}

\makeatletter

\def\jobis#1{FF\fi
  \def\predicate{#1}%
  \edef\predicate{\expandafter\strip@prefix\meaning\predicate}%
  \edef\job{\jobname}%
  \ifx\job\predicate
}

\makeatother

\if\jobis{proposal}%
\else
\fi

 \usepackage[matrix, arrow]{xy}
 \renewcommand{\mod}{\ \operatorname{mod}}
\DeclareMathOperator{\Bs}{Bs}
\DeclareMathOperator{\Supp}{Supp}

\DeclareMathOperator{\Fix}{Fix}

\DeclareMathOperator{\Mov}{Mov}

\DeclareMathOperator{\Spec}{Spec}

 \numberwithin{equation}{subsection}
 \numberwithin{footnote}{subsection}

 \newtheorem{lem}[subsection]{Lemma}
 
 \newtheorem{thm}[subsection]{Theorem}

{
    \newtheoremstyle{upright}%
        {8pt plus2pt minus4pt}%
        {8pt plus2pt minus4pt}%
        {\upshape}%
        {}%
        {\bfseries\scshape}%
        {}%
        {1em}%
        {}%
\theoremstyle{upright}

 \newtheorem{defn}[subsection]{Definition}
 
 \newtheorem{exa}[subsection]{Example}

 \newtheorem{rem}[subsection]{Remark}

}

 \newcommand{\x}{\mathscr}
 \newcommand{\C}{\mathbb C}
 
 \newcommand{\PP}{\mathbb P}
 
 \newcommand{\Q}{\mathbb Q}
 
 \newcommand{\Z}{\mathbb Z}
 \newcommand{\bir}{\dashrightarrow}
 \newcommand{\rddown}[1]{\left\lfloor{#1}\right\rfloor} 

\title{Divisorial algebras and modules on schemes}
\author{Caucher Birkar}
\date{\today}
\begin{document}
\maketitle

\begin{abstract}
We study certain modules over the algebra of a Cartier divisor on a scheme. 
Using these modules, we present an inductive method
for studying finite generation properties of algebras and modules. In the context of the 
minimal model program, we show that 
finite generation of log canonical algebras and modules is equivalent to 
the minimal model and abundance conjectures. 
\end{abstract}

\tableofcontents


\section{Divisorial algebras and modules}

Let $X$ be a projective scheme over a (commutative) Noetherian ring $A$. For any Cartier divisor $L$ on $X$ 
we have the graded ring 
$$
{R}(L):=\bigoplus_{m\ge 0} H^0(X,mL)
$$ where 
$m$ runs through the non-negative integers. This is actually a graded algebra over the 
ring $R_0:=H^0(X,\x{O}_X)=\x{O}_X(X)$. The structure morphism $X\to \Spec A$ gives a 
canonical ring homomorphism $A\to R_0$ making $R(L)$ an algebra over $A$. 
Since 
$R_0$ is a finitely generated $A$-module, ${R}(L)$ is a finitely 
generated $R_0$-algebra iff it is a finitely generated $A$-algebra. We refer to  
$R(L)$ as a \emph{divisorial algebra}.

The problem of finite generation of divisorial algebras is  fundamental 
 in algebraic geometry and we will see that divisorial modules naturally appear when one attempts 
to approach this problem. We study this problem in a very general setting and we are ultimately 
interested in the birational geometry of schemes. However, our methods are very relevant to 
the traditional setting of birational geometry of varieties over an algebraically closed field. 

 For each $\x{O}_X$-module $\x{F}$ on $X$ and each integer $p$, we have the 
graded ${R}(L)$-module ${M}_{\x{F}}^p(L)=\bigoplus_{m\in \Z}M_m$ where $M_m=0$ if $m<p$ but 
$$
M_m= H^0(X,\x{F}(mL))
$$
if $m\ge p$. Here $\x{F}(mL)$ stands for $\x{F}\otimes_{\x{O}_X} \x{O}_X(mL)$ and 
the module structure is given via the pairing
$$
H^0(X,mL) \otimes H^0(X,\x{F}(nL)) \to
  H^0(X,\x{F}((m+n)L))
$$
We refer to  ${M}_{\x{F}}^p(L)$ as a \emph{divisorial module}.  
When $\x{F}=\x{O}_X(D)$ for some divisor $D$ we often write ${M}_{D}^p(L)$
instead of ${M}_{\x{O}_X(D)}^p(L)$.
When $L=I(K_X+B)$ for a log canonical pair $(X,B)$ and integer $I>0$, we refer to ${R}(L)$ 
as a \emph{log canonical algebra} and refer to the module  
${M}_{\x{F}}^p(L)$ as a \emph{log canonical module}.

The above graded rings and modules are of course nothing new. Grothendieck 
studied them in some detail (for example see EGA III-1 [\ref{EGA}, \S 2]). 
Strictly speaking he considers all values $m\in \Z$ but in practice 
the summands $H^0(X,mL)$ and $H^0(X,\x{F}(mL))$ usually vanish for $m\ll 0$, eg 
when $L$ is ample. However, we look at divisorial modules in the sense of birational 
geometry, in particular, their finite generation properties. So, it
makes sense to restrict ourselves to non-negative $m$ or at least $m\ge p$ for some fixed $p$. 
We should remark that we could define algebras and modules using an invertible 
sheaf $\mathcal{L}$ in place of a Cartier divisor. Some of the results would still make sense 
and will be true in this setting (eg, Theorem 1.1).\\ 

{\textbf{Results of this paper.}} 
In section 2, we prove some of the basic properties of divisorial modules, in particular:

\begin{thm}
Let $X$ be a projective scheme over a Noetherian ring $A$ and $L$ a Cartier divisor 
on $X$ such that ${R}(L)$ is a finitely generated $A$-algebra. We fix an integer $p$ and 
an invertible sheaf 
$\x{O}_X(1)$ very ample over $\Spec A$. Then we have:

$\bullet$
Assume that ${M}_{\x{O}_X(l)}^p(L)$ is a finitely generated ${R}(L)$-module for any $l\gg 0$. 
Then ${M}_{\x{F}}^p(L)$ is a finitely generated ${R}(L)$-module for any 
reflexive coherent sheaf $\x{F}$. If $X$ is integral, the same holds for any 
torsion-free coherent sheaf $\x{F}$.

$\bullet$ 
Let $\x{F}$ be a coherent sheaf and $I>0$ an integer. 
For each $0\le i<I$, assume that ${M}_{\x{F}(iL)}^{q_i}(IL)$ is a finitely generated ${R}(IL)$-module 
where $q_i\in \Z$ is the smallest number satisfying $q_iI+i\ge p$. 
Then ${M}_{\x{F}}^p(L)$ is a finitely generated ${R}(L)$-module.
\end{thm}

In section 3, we give necessary and sufficient conditions for algebras and modules to be finitely 
generated in terms of finite generation of restriction of algebras and modules 
to subschemes:

\begin{thm}
Let $X$ be a projective scheme over a Noetherian ring $A$ and $L=\sum l_iL_i$ a Cartier divisor
where $l_i> 0$ are integers and $L_i\neq 0$ are effective Cartier divisors. Assume  
that $H^0(X,-L_1)=0$. Fix an invertible sheaf $\x{O}_X(1)$ very ample over $\Spec A$. 
Then, the following are equivalent:

$\bullet$ $R(L)$ is a finitely generated $A$-algebra and $M^p_{\x{F}}(L)$ is a finitely generated 
$R(L)$-module for any reflexive coherent sheaf $\x{F}$ and any $p$;

$\bullet$ for each $S=L_j$, the restriction $R(L)|_{S}$ is a finitely generated $A$-algebra, and  
 the restriction $M^0_{\x{O}_X(l)}(L)|_S$ is a finitely generated $R(L)|_S$-module 
for any $l\gg 0$. 
\end{thm}

Using this theorem and the available extension theorems (eg, Hacon-M$\rm ^c$Kernan) it should 
not be too difficult (but not simple) to give another proof of finite generation of 
lc rings of klt pairs even without the minimal model program. We hope that the theorem 
will be useful for proving finite generation of lc rings of lc pairs.
  
In section 4, we show that finite generation of log canonical algebras and modules 
is closely related to the minimal model and abundance conjectures:

\begin{thm}
Assume that $(X/Z,B)$ is lc and assume that there is a positive 
integer $I$ such that $L:=I(K_X+B)$ is Cartier and pseudo-effective$/Z$ 
where $Z=\Spec A$. 
Then, $(X/Z,B)$ has a log minimal model $(Y/Z,B_Y)$ on 
which $K_Y+B_Y$ is semi-ample$/Z$ iff 

$\bullet$ $R(L)$ is a finitely generated $A$-algebra, and 

$\bullet$ for any very ample$/Z$ divisor $G$ 
the module $M_{G}^0(L)$ is finitely generated over $R(L)$.\\ 
\end{thm}

{\textbf{Conventions.}} 
In this paper, all rings are commutative with identity. A graded ring is of the form 
$R=\bigoplus_{m\ge 0}R_m$, that is, graded by non-negative integers, and 
a graded module is of the form ${M}=\bigoplus_{m\in \Z}M_m$, that is, it is 
graded by the integers. For an element $(\dots,\alpha,\dots)$ of degree $m$ we often 
abuse notation and just write $\alpha$ but keep in mind that $\alpha$ has degree $m$.
A graded ring $R=\bigoplus_{m\ge 0}R_m$ is called an $A$-algebra if there is a homomorphism 
$A\to R_0$ and multiplication of elements of $A$ with elements of $R$ is induced by $A\to R_0$.

\vspace{0.5cm}
\section{Basic properties of divisorial modules}

\begin{defn}
If $R=\bigoplus_{m\ge 0} R_m$ is a graded ring and $I$ a positive integer, we define 
the truncated ring $R^{[I]}=\bigoplus_{m\ge 0} R_m'$ as $R'_m=R_m$ if $I|m$ and $R_m'=0$ 
otherwise. Note that the degree structure is different from the usual definition of 
truncation. However, it is more convenient for us to define it in this way.
\end{defn}

\begin{rem}[Truncation principle for algebras]
Fix a positive integer $I$. Assume that $R$ is a graded ring, $R_0$ is a Noetherian ring, 
and $R$ is an integral domain. Then $R$ is a finitely generated $R_0$-algebra 
iff  $R^{[I]}$ is a finitely generated $R_0$-algebra. 
\end{rem}

A similar statement for modules is less straightforward.

\begin{rem}[Truncation principle for modules]\label{r-lcm-truncation}
Fix a positive integer $I$. Let $R$ be a graded ring 
and let $M=\bigoplus_{m\in \Z} M_m$ be a graded $R$-module. 
Let $N_i=\bigoplus_{m\in \Z} N_{m,i}$ where $N_{m,i}=M_m$ if $m\equiv i (\mod I)$ but 
$N_{m,i}=0$ otherwise. Then, each $N_i$ is a graded module over $R^{[I]}$ and 
we have the decomposition 
$$
M\simeq N_0\oplus N_1\oplus \cdots \oplus N_{I-1}
$$ 
as graded $R^{[I]}$-modules.
If the modules $N_0,\dots,N_{I-1}$ are finitely generated over $R^{[I]}$, 
then $M$ is also a finitely generated $R^{[I]}$-module hence a finitely 
generated $R$-module too.
\end{rem} 

\begin{lem}\label{r-change-p}
Let $X$ be a projective scheme over a Noetherian ring $A$. Assume that ${M}_{\x{F}}^p(L)$ is a finitely 
generated $R(L)$-module where $\x{F}$ is coherent. Then 
${M}_{\x{F}}^q(L)$ is also a finitely generated $R(L)$-module for any $q<p$.
If in addition $R(L)$ is a finitely generated $A$-algebra, then 
${M}_{\x{F}}^q(L)$ is a finitely generated $R(L)$-module for any $q$.
\end{lem}
\begin{proof}
The elements of ${M}_{\x{F}}^q(L)$ of degree $<p$ are given by 
$$
H^0(X,\x{F}(qL))\oplus \cdots \oplus H^0(X,\x{F}((p-1)L))
$$ 
and this is a finitely generated $R_0=H^0(X,\x{O}_X)$-module as $\x{F}$ is coherent.
The elements of ${M}_{\x{F}}^q(L)$ of degree $\ge p$ are given by ${M}_{\x{F}}^p(L)$.
So, the first claim follows.
 The second statement follows from the fact that if $p\le q$, then we have an  
inclusion ${M}_{\x{F}}^q(L)\subseteq {M}_{\x{F}}^p(L)$, 
and $R(L)$ and ${M}_{\x{F}}^p(L)$ are Noetherian.\\
\end{proof}

\begin{thm}\label{t-lmc-ample-to-general}
Let $X$ be a projective scheme over a Noetherian ring $A$ and $L$ a Cartier divisor 
on $X$ such that ${R}(L)$ is a finitely generated $A$-algebra. We fix an integer $p$ and 
an invertible sheaf 
$\x{O}_X(1)$ very ample over $\Spec A$. Then we have:

$(1)$
Assume that ${M}_{\x{O}_X(l)}^p(L)$ is a finitely generated ${R}(L)$-module for any $l\gg 0$. 
Then ${M}_{\x{F}}^p(L)$ is a finitely generated ${R}(L)$-module for any 
reflexive coherent sheaf $\x{F}$. If $X$ is integral, the same holds for any 
torsion-free coherent sheaf $\x{F}$.

$(2)$ 
Let $\x{F}$ be a coherent sheaf and $I>0$ an integer. 
For each $0\le i<I$, assume that ${M}_{\x{F}(iL)}^{q_i}(IL)$ is a finitely generated ${R}(IL)$-module 
where $q_i\in \Z$ is the smallest number satisfying $q_iI+i\ge p$. 
Then ${M}_{\x{F}}^p(L)$ is a finitely generated ${R}(L)$-module.
\end{thm}
\begin{proof}
(1)  The first statement: there is a surjective morphism 
$\bigoplus_{j=1}^r \x{O}_X(-l_j)\to \x{F}^{\vee}$ 
where $l_j\gg 0$ and $\vee$ stands for dual. Taking the dual of this 
morphism gives an injective morphism  
$$
\x{F}\simeq \x{F}^{\vee\vee} \to \x{E}=\bigoplus_{j=1}^r \x{O}_X(l_j)
$$
which in turn gives  an injective map ${M}^p_{\x{F}}(L)\to {M}^p_{\x{E}}(L)$. 
By assumptions,  
${M}^p_{\x{E}}(L)$ is finitely generated over ${R(L)}$ which in particular means that 
${M}^p_{\x{E}}(L)$ is Noetherian as $R(L)$ is Noetherian. Therefore, each submodule of ${M}^p_{\x{E}}(L)$ 
is also finitely generated over ${R}(L)$, in particular, ${M}^p_{\x{F}}(L)$.

The second statement: since $X$ is integral, if $\x{F}$ is any torsion-free 
coherent sheaf, the natural morphism $\x{F}\to \x{F}^{\vee\vee}$ is injective 
(cf. [\ref{Hartshorne}][\ref{Schwede}, Lemma 2.5]). 
So, we get an injective map $M^p_\x{F}(L)\to M^p_{\x{F}^{\vee\vee}}(L)$ and the 
claim follows from the first statement as $\x{F}^{\vee\vee}$ is a reflexive sheaf.

(2) We can write ${M}_{\x{F}}^p(L)\simeq N_0\oplus N_1\oplus \cdots \oplus N_{I-1}$ as in Remark 
\ref{r-lcm-truncation}. 
 Let $N_i'$ be the module over 
$R(IL)$ whose $n$-th degree summand is just $N_{nI+i,i}=M_{nI+i}$. 
In fact, $N_i'=M_{\x{F}(iL)}^{q_i}(IL)$ where $q_i\in \Z$ is the smallest number 
satisfying $q_iI+i\ge p$.
By assumptions, $N_i'$ is a finitely generated $R(IL)$-module. Therefore, 
$N_i$ is a finitely generated $R(L)^{[I]}$-module and we can use Remark \ref{r-lcm-truncation}. 
Note that the degree $m$ elements of $R(IL)$ are the same as the degree $mI$ elements of 
$R(L)^{[I]}$, and the degree $n$ elements of $N_i'$ are the same as the degree $nI+i$ elements of 
$N_i$.\\
\end{proof}

\begin{exa}
In this example, we show that Theorem \ref{t-lmc-ample-to-general} (1) does not hold if 
we drop the relfexive and torsion-free properties. 
Take a Lefschetz pencil on $\PP^n_\C$ where $n\ge 2$, 
blow up the base locus to get $g\colon X\to \PP^n_\C$, and 
let $f\colon X\to \PP^1_\C$ 
be the corresponding Lefschetz fibration. 
Let $G$ be a very ample divisor on $X$ and 
$L$ the effective exceptional$/\PP^n_\C$ divisor satisfying $g^*g_*G=G+L$. 
Since $L$ is exceptional and effective, we have 
$$
R(L)=\C\oplus \C \oplus \cdots 
$$
which is a finitely generated $\C$-algebra. Moreover, for each $m>0$, 
$$
H^0(X,mL+G)=H^0(X,(m-1)L+g^*g_*G)=H^0(X,g^*g_*G)
$$
which is independent of $m$ hence $M^0_{G}(L)$ is a finitely generated $R(L)$-module. 

On the other hand, if $F$ is a general fibre of $f$, then $g^*g_*F=F+E$ for some 
effective exceptional divisor $E$ so $E|_F=g^*g_*F|_F$ is a big divisor. The support 
of $L$ contains the support of $E$ hence $L|_F$ is also a big divisor.
Now put $\x{F}=\x{O}_F$ which is not reflexive nor torsion-free. Then, for each $m$ we have 
$$
H^0(X,\x{F}(mL))=H^0(F,\x{O}_F(mL))
$$
and its  dimension grows like $m^{n-1}$ as $L|_F$ is big. This implies that $M^0_{\x{F}}(L)$ 
is not a finitely generated $R(L)$-module. In fact, fix a positive integer $e$ and let 
$N$ be the submodule of $M^0_{\x{F}}(L)$ 
generated by the elements of degree $\le e$. Then, for any $m\gg 0$, the degree $m$  
homogeneous piece of $N$ is a $\C$-vector space of dimension at most 
$$
h^0(F,\x{O}_F)+h^0(F,\x{O}_F(L))+\dots+h^0(F,\x{O}_F(eL))
$$ 
which cannot grow like $m^{n-1}$.\\
\end{exa}

\begin{thm}\label{t-lcm-big-divisors}
Let $X$ be a projective scheme over a Noetherian ring $A$ and assume that 
$L$ is a Cartier divisor such that $R(L)$ is a finitely generated $A$-algebra. 
Moreover, suppose that $IL\sim E+G$ for some positive integer $I$ where $E$ is effective  
and $G$ is ample over $\Spec A$. Then, $M_{\x{F}}^p(L)$ is a finitely generated $R(L)$-module 
for every $p$ and every reflexive coherent sheaf $\x{F}$. If $X$ is integral, the same 
holds for any torsion-free coherent sheaf $\x{F}$.
\end{thm}
\begin{proof}
We may assume that $G$ is very ample over $\Spec A$.
By Theorem \ref{t-lmc-ample-to-general} (1), it is enough to verify the finite generation 
of ${M}_{lG}^p(L)$ where $l\gg 0$. 
Since $lIL\sim lE+lG$ and $lE$ is effective, there is an injective map 
$$
{M}_{lG}^p(L)\to M_{lE+lG}^p(L)\simeq M_{lIL}^p(L)
$$ 
 By Lemma \ref{r-change-p}, it is 
enough to show that ${M}_{lIL}^{-lI}(L)$ is a finitely generated $R(L)$-module. 
Now the elements of degree $-lI$ are $H^0(X,-lIL+lIL)=H^0(X,\x{O}_X)$ which 
contains $1\in \x{O}_X(X)$. 
If $\alpha\in {M}_{lIL}^{-lI}(L)$ 
is a homogeneous element of degree $m\ge -lI$, that is, an element of $H^0(X,mL+lIL)$, then 
$\alpha=\alpha\cdot 1$ where we consider the second $\alpha$ as an element of 
$R(L)$ of degree $m+lI$ and we consider $1$ as an element of  ${M}_{lIL}^{-lI}(L)$ 
of degree $-lI$. So,  ${M}_{lIL}^{-lI}(L)$ is generated over $R(L)$ by the element 
$1$ of degree $-lI$.
\end{proof}

\vspace*{0.3cm}

\section{Inductive finite generation of divisorial algebras and modules}

 Let $X$ be a projective scheme over a Noetherian ring 
$A$, and let $L$ be a Cartier divisor on $X$.
One of the main ideas that one might use to prove that $R(L)$ is a finitely generated $A$-algebra, 
is by induction on dimension, that is, restriction to subschemes.
For each closed subscheme $S$ of $X$ we have the exact sequence 
$$
0\to \x{I}_S(mL) \to \x{O}_X(mL) \to \x{O}_S(mL) \to 0
$$
which gives the exact sequence 
$$
0\to \bigoplus_{m\ge 0} H^0(X,\x{I}_S(mL)) \to R(L)=\bigoplus_{m\ge 0} H^0(X,\x{O}_X(mL)) 
\xrightarrow{\phi} \bigoplus_{m\ge 0} H^0(S,\x{O}_S(mL))
$$
We can try to use such sequences for the induction process. However, we face two main issues here: 
we need the finite generation of the image of $\phi$ and we need to relate the kernel of $\phi$ 
to the finite generation of $R(L)$. If we denote the image of $\phi$ by 
$R(L)|_S$, then we have the exact sequence 
$$
0\to M^0_{\x{I}_S}(L) \to R(L) \to R(L)|_S \to 0 
$$
In this section, we show that the theory of divisorial modules provides a convenient 
way for dealing with the kernel issue. The main idea is to consider $M^0_{\x{I}_S}(L)$ 
as an $R(L)$-module. Assume that $M^0_{\x{I}_S}(L)$ has no non-trivial elements of degree zero, 
i.e. $H^0(X,\x{I}_S)=0$. Then, $R(L)$ is a finitely generated $A$-algebra 
iff $R(L)|_S$ is a finitely generated $A$-algebra and $M^0_{\x{I}_S}(L)$ is a finitely 
generated $R(L)$-module (see Lemma \ref{l-modules-to-algebras}).
Needless to say, if we take $S$ to be arbitrary we should not expect to get much 
information from the above sequences. Instead, one has to choose the $S$ carefully.

Now let $\x{F}$ be an $\x{O}_X$-module. The exact sequence 
$$
 \x{F}\otimes \x{I}_S(mL) \to \x{F}\otimes\x{O}_X(mL) \to \x{F}\otimes\x{O}_S(mL) \to 0
$$
induces a map 
$$
H^0(X,\x{F}\otimes\x{O}_X(mL)) \to H^0(X,\x{F}\otimes\x{O}_S(mL))
$$
hence a map $M^p_{\x{F}}(L)\to M^p_{\x{F}\otimes\x{O}_S}(L)$ of $R(L)$-modules. 
We denote the image of the latter 
map by $M^p_{\x{F}}(L)|_S$ and call it the restriction of $M^p_{\x{F}}(L)$ to $S$. 
Note that $M^p_{\x{F}}(L)|_S$ is a module over $R(L)|_S$. If $\x{F}$ is locally free, then 
we have an exact sequence 
$$
0\to M^p_{\x{F}\otimes \x{I}_S}(L)\to M^p_{\x{F}}(L)\to M^p_{\x{F}}(L)|_S \to 0
$$
Now assume in addition that $S=L$ is an effective Cartier divisor.
Then, $\x{O}_X(-L)\simeq \x{I}_L$ hence there is an injective 
morphism  $\x{O}_X(-L)\to \x{O}_X$ whose image is $\x{I}_L$. We then get an injective morphism 
$\x{O}_X\to \x{O}_X(L)$.
 Let $\alpha\in H^0(X,L)$ be the image of $1\in H^0(X,\x{O}_X)$.  
Then, $\alpha$ is a homogeneous element of $R(L)$ of degree one, and in the above 
exact sequence we can replace $M^p_{\x{F}\otimes \x{I}_S}(L)\to M^p_{\x{F}}(L)$ with 
the  map $ M^p_{\x{F}(-L)}(L)\to M^p_{\x{F}}(L)$ which 
is given by multiplication with $\alpha$.\\

\begin{lem}\label{l-modules-to-algebras}
Let $A$ be a Noetherian ring and let 
$$
\phi\colon R=\bigoplus_{m\ge 0}R_m\to T=\bigoplus_{m\ge 0}T_m
$$ 
be a surjective graded homomorphism of graded $A$-algebras. Assume further that $R_0\to T_0$ is 
an isomorphism. Let $K=\ker \phi$.
Then, $R$ is a finitely generated $A$-algebra iff $T$ is a finitely generated $A$-algebra 
and $K$ is a finitely generated $R$-module.  
\end{lem}
\begin{proof}
If $R$ is a finitely generated $A$-algebra, then the claim is obvious as $R$ would be 
Noetherian. Conversely, assume that $T$ is a finitely generated $A$-algebra 
and $K=\bigoplus_{m\ge 0}K_m$ is a finitely generated $R$-module. We may assume that $K\neq 0$.
Let $\alpha_1, \dots, \alpha_r$ be homogensous elements 
of $R$ mapping to a set of generators of $T$ over $A$, and let $\beta_1, \dots,\beta_s$ 
be a set of non-zero homogeneous generators of $K$ over $R$. Since $R_0\to T_0$ is an isomorphism, 
$K_0=0$, hence $\deg \beta_i>0$ for every $i$. Moreover, $R_0\simeq T_0$ is a finitely generated 
$A$-module hence we could assume that $A=R_0$.

We will prove that $R$ is generated as an $A$-algebra by the elements $\alpha_i$ and $\beta_j$.
We prove the claim by induction on degree. Pick $\gamma\in R$ homogeneous 
of degree $m$ and assume that the claim is true for elements of degree $<m$. 
Then, $\gamma=\gamma_1+\gamma_2$ where $\gamma_i$ are  homogeneous of degree $m$, 
$\gamma_1$ is a polynomial over $A$ in the 
$\alpha_i$, and $\gamma_2\in K$. We can write 
$\gamma_2=\sum a_i\beta_i$ where $a_i\in R$ are homogeneous and $\deg a_i\beta_i=m$ for each $i$. 
 In particular,  $\deg a_i<m$ because $\deg \beta_i>0$. By induction, each $a_i$ 
 is a polynomial over $A$ in the $\alpha_i$ and $\beta_j$. Therefore, the same holds for 
 $\gamma$.\\ 
\end{proof}

\begin{lem}\label{l-fg-exact-sequence}
Let $R$ be a ring and $0\to M'\to M \to M''\to 0$ an exact sequence of $R$-modules.
If $M'$ and $M''$ are finitely generated over $R$, then $M$ is also finitely generated 
over $R$.  
\end{lem}
\begin{proof}
We can think of $M'$ as a submodule of $M$. 
Let $\alpha_1,\dots,\alpha_r$ be elements of $M$ whose images generate $M''$ over $R$, 
and let $\beta_1,\dots,\beta_t$ be generators of $M'$. Each $\gamma\in M$ can written 
as $\gamma=\gamma_1+\gamma_2$ where $\gamma_2$ belongs to the submodule generated 
by the $\alpha_i$ and $\gamma_2\in M'$. But then $\gamma$ belongs to the submodule 
generated by the $\alpha_i$ and $\beta_j$.\\ 
\end{proof}

\begin{thm}\label{t-dm-1}
Let $X$ be a projective scheme over a Noetherian ring $A$ and $L=\sum l_iL_i$ a Cartier divisor
where $l_i> 0$ are integers and $L_i\neq 0$ are effetive Cartier divisors. 
Let $\x{F}$ be a locally free coherent sheaf on $X$. Assume that, for each $S=L_j$, the restriction 
$M^p_{\x{F}(-C)}(L)|_S$ is a finitely generated $R(L)|_S$-module for any 
$C=\sum c_iL_i$ where $c_i\in[0,l_i]$ are integers and $c_j<l_j$. Then, $M^p_{\x{F}}(L)$ is a 
finitely generated $R(L)$-module.
\end{thm}
\begin{proof}
Let $C=\sum c_iL_i$ where $0\le c_i\le l_i$ are integers. 
Assume that in the exact sequence 
$$
0\to M^p_{\x{F}(-C)}(L) \to M^p_{\x{F}}(L) \to M^p_{\x{F}}(L)|_C \to 0
$$
the module $M^p_{\x{F}}(L)|_C$ is a finitely generated $R(L)|_C$-module. 
The surjection $R(L)\to R(L)|_C$ makes ${M^p_{\x{F}}(L)}|_C$ into a finitely 
generated $R(L)$-module.
We increase $C$ inductively so we could have started with $C=0$ at the first step.
Assume that $C\neq L$ and let $S=L_j$ such that $c_j< l_j$.
We have an exact sequence 
$$
0\to M^p_{\x{F}(-C-S)}(L) \to M^p_{\x{F}(-C)}(L) \to M^p_{\x{F}(-C)}(L)|_S \to 0
$$
where by assumptions $M^p_{\x{F}(-C)}(L)|_S$ is a finitely generated $R(L)|_S$-module 
hence a finitely generated $R(L)$-module. 
By induction and by Lemma \ref{l-fg-exact-sequence}, the module in the middle in the 
exact sequence 
$$
0\to  \frac{M^p_{\x{F}(-C)}(L)}{M^p_{\x{F}(-C-S)}(L)} \to \frac{M^p_{\x{F}}(L)}{M^p_{\x{F}(-C-S)}(L)}  
 \to \frac{M^p_{\x{F}}(L)}{M^p_{\x{F}(-C)}(L)}    \to 0
$$
is finitely generated over $R(L)$ which in turn 
 implies that in the exact sequence 
$$
0\to M^p_{\x{F}(-C-S)}(L) \to M^p_{\x{F}}(L) \to M^p_{\x{F}}(L)|_{C+S} \to 0
$$
the module $M^p_{\x{F}}(L)|_{C+S}$ is a finitely generated $R(L)$-module.
 By continuing this process we reach the 
situation $C+S=L$. In particular, we deduce that in the exact sequence  
$$
0\to {M^p_{\x{F}(-L)}(L)}\to {M^p_{\x{F}}(L)} \to {M^p_{\x{F}}(L)}|_L\to 0
$$
the module ${M^p_{\x{F}}(L)}|_L$ is a finitely generated $R(L)|_L$-module hence a 
finitely generated $R(L)$-module. 
 Suppose that it is generated by the images of homogeneous 
elements $\gamma_1,\dots,\gamma_r\in M^p_{\x{F}}(L)$. Let $N$ be the submodule of ${M^p_{\x{F}}(L)}$
generated by these elements. 

Pick a homogeneous element 
$\beta\in {M^p_{\x{F}}(L)}$ of degree $m$. Then, we can write $\beta=\beta_1+\beta_2$ 
where $\beta_1$ is homogeneous of degree $m$ in $N$, and $\beta_2$ 
is in the image of the  map $ M^p_{\x{F}(-L)}(L)\to M^p_{\x{F}}(L)$. But this map is 
nothing but multiplication with $\alpha$ for some $\alpha\in H^0(X,L)$. 
Thus, $\beta_2=\alpha\cdot \beta_3$ where $\beta_3$ is a homogeneous element of $ M^p_{\x{F}(-L)}(L)$
of degree $m$, that is, 
$$
\beta_2\in H^0(X,\x{F}(-L)(mL))\simeq H^0(X,\x{F}((m-1)L))
$$ 
Since the degree $m-1$ 
elements of $M^p_{\x{F}}(L)$ are just $H^0(X,\x{F}((m-1)L))$, we can consider 
$\beta_3$ as an element of $ M^p_{\x{F}}(L)$ of degree $m-1$.
We can repeat this process with $\beta_3$ in place of $\beta$ 
and at the end write  $\beta$ as $\omega_1+\omega_2$ 
where $\omega_1\in N$ and $\omega_2$ has 
bounded degree. This implies the claimed finite generation.\\
\end{proof}

\begin{thm}\label{c-dm-2}
Let $X$ be a projective scheme over a Noetherian ring $A$ and $L=\sum l_iL_i$ a Cartier divisor
where $l_i> 0$ are integers and $L_i\neq 0$ are effective Cartier 
divisors. Assume further that $H^0(X,-L_1)=0$. 
Then, the following are equivalent:

$(1)$ $R(L)$ is a finitely generated $A$-algebra;

$(2)$ $R(L)|_{L_1}$  is a finitely generated $A$-algebra, and for each $S=L_j$, 
 the restriction $M^0_{-C}(L)|_S$ is a finitely generated $R(L)|_S$-module for any 
 $C=\sum c_iL_i$ where $c_i\in [0, l_i]$ are integers and $c_j<l_j$. 
\end{thm}
\begin{proof}
(1) $\implies$ (2): Since $R(L)$ is a finitely generated $A$-algebra, it is a Noetherian 
ring. Obviously, each restriction $R(L)|_S$ is a finitely generated $A$-algebra. On the 
other hand, for each $C$ as in the theorem, 
$M^0_{-C}(L)$ is isomorphic to an ideal of $R(L)$ hence it is a finitely 
generated $R(L)$-module which in turn implies that $M^0_{-C}(L)|_S$ is a finitely generated 
$R(L)|_S$-module.

(2) $\implies$ (1):  Let $S=L_1$ and consider the exact sequence 
$$
0\to M^0_{-S}(L) \to R(L) \to R(L)|_S \to 0
$$  
Since $H^0(X,-L_1)=0$, $H^0(X,0)\to H^0(S,0)$ is injective hence $R(L) \to R(L)|_S$ 
is an isomorphism in degree zero. 
Thus, by Lemma \ref{l-modules-to-algebras}, $R(L)$ is a finitely generated $A$-algebra 
iff $M^0_{-S}(L)$ is a finitely generated $R(L)$-module because we have assumed that 
$R(L)|_S$ is a finitely generated $A$-algebra. For the rest we argue similar to the 
proof of Theorem \ref{t-dm-1}. 

Let $C=\sum c_iL_i$ where $0\le c_i\le l_i$ are integers. 
Assume that we have already proved that  
$M^0_{-L_1}(L)$ is a finitely generated $R(L)$-module iff $M^0_{-C}(L)$ is a 
finitely generated $R(L)$-module.
We increase $C$ inductively so we could have started with $C=L_1$ at the first step.
Assume that $C\neq L$ and let $S=L_j$ such that $c_j< l_j$.
We have an exact sequence 
$$
0\to M^0_{-C-S}(L) \to M^0_{-C}(L) \to M^0_{-C}(L)|_S \to 0
$$
where by assumptions $M^0_{-C}(L)|_S$ is a finitely generated $R(L)|_S$-module hence a 
finitely generated $R(L)$-module. 
Thus, $M^0_{-L_1}(L)$ is a finitely generated $R(L)$-module  iff $M^0_{-C}(L)$ is a 
finitely generated $R(L)$-module iff $M^0_{-C-S}(L)$ is a 
finitely generated $R(L)$-module. We continue this until we reach the situation 
$L=C+S$. In that case, $M^0_{-L_1}(L)$ is a finitely generated $R(L)$-module iff 
$M^0_{-L}(L)$ is a finitely generated $R(L)$-module. But $M^0_{-L}(L)$ is generated 
by the element $1\in M^0_{-L}(L)$ of degree one.\\
\end{proof}

\begin{thm}\label{t-dm-3}
Let $X$ be a projective scheme over a Noetherian ring $A$ and $L=\sum l_iL_i$ a Cartier divisor
where $l_i> 0$ are integers and $L_i\neq 0$ are effective Cartier divisors. Assume  
that $H^0(X,-L_1)=0$. Fix an invertible sheaf $\x{O}_X(1)$ very ample over $\Spec A$. 
Then, the following are equivalent:

$(1)$ $R(L)$ is a finitely generated $A$-algebra and $M^p_{\x{F}}(L)$ is a finitely generated 
$R(L)$-module for any reflexive coherent sheaf $\x{F}$ and any $p$;

$(2)$ for each $S=L_j$, the restriction $R(L)|_{S}$ is a finitely generated $A$-algebra, and  
 the restriction $M^0_{\x{O}_X(l)}(L)|_S$ is a finitely generated $R(L)|_S$-module 
for any $l\gg 0$. 
\end{thm}
\begin{proof}
(1) $\implies$ (2): Obvious. We prove the converse. Let $\x{P}$ be any coherent locally free 
sheaf on $X$. There is a surjective morphism 
$\pi\colon \bigoplus_{1}^n \x{O}_X(-l_i) \to \x{P}^{\vee}$ 
where $l_i\gg 0$. Put $\x{E}_1=\bigoplus_{1}^n \x{O}_X(l_i)$. The kernel of $\pi$ is 
also locally free as $\x{P}^{\vee}$ and $\x{E}_1^{\vee}$ are both locally free.  
Assume that the kernel is $\x{E}_2^\vee$. Then, by dualising we get an exact 
sequence 
$$
0\to \x{P}\to \x{E}_1 \to \x{E}_2 \to 0
$$
 If $\x{G}$ is any coherent sheaf on 
$X$, then the sequence 
$$
0\to \x{P} \otimes_{\x{O}_X} \x{G}\to \x{E}_1\otimes_{\x{O}_X} \x{G} \to \x{E}_2\otimes_{\x{O}_X} \x{G} \to 0
$$
is again exact (cf. [\ref{B-topics}, Theorem 3.1.8]).

Now let $S=L_j$ for some $j$. Then, the horizontal arrows in the commutative diagram
$$
\xymatrix{
\x{P}(mL) \ar[d] \ar[r] & \x{E}_1(mL)\ar[d] \\
\x{P}(mL) \otimes_{\x{O}_X} \x{O}_S \ar[r] & \x{E}_1(mL)\otimes_{\x{O}_X} \x{O}_S
}
$$ 
are injective hence $M^0_{\x{P}}(L)|_S$ is a submodule of $M^0_{\x{E}_1}(L)|_S$. 
By our assumptions, $M^0_{\x{E}_1}(L)|_S$ is a finitely generated $R(L)|_S$-module. 
Therefore, $M^0_{\x{P}}(L)|_S$ is also a finitely generated $R(L)|_S$-module as 
$R(L)|_S$ is Noetherian by assumptions. Now simply apply Theorem \ref{c-dm-2} 
and Theorem \ref{t-dm-1} to prove that $R(L)$ is a finitely generated $A$-algebra 
and to show that $M^0_{\x{F}}(L)$ is a finitely generated $R(L)$-module for $\x{F}$ 
locally free, in particular, 
for $\x{F}$ of the form $\x{O}_X(l)$. Finally, apply Theorem \ref{t-lmc-ample-to-general} 
and Lemma \ref{r-change-p} to get the claim for any reflexive coherent $\x{F}$ and any $p$.\\
\end{proof}

\begin{rem}\label{rem-restriction-to-L} 
In Theorem \ref{t-dm-1}, 
we use the successive restrictions to the various $L_i$ in order to prove that 
in the exact sequence 
$$
0\to M^p_{\x{F}(-L)}(L) \to M^p_{\x{F}}(L) \to M^p_{\x{F}}(L)|_L \to 0
$$
the restriction $M^p_{\x{F}}(L)|_L$ is a finitely generated $R(L)|_L$-module. 
In practice, it might be easier to verify finite generation of the restrictions 
on the $L_i$ rather than directly on $L$, eg when $L_i$ are prime divisors on a 
normal variety $X$ over an algebraically closed field. So, it is crucial to allow the 
extra flexibility.
Similarly, in Theorem \ref{c-dm-2}, we could use the exact sequence 
$$
0\to M^0_{-L}(L) \to R(L) \to R(L)|_L \to 0
$$
where if $R(L)|_L$ is a finitely generated $A$-algebra, then  $R(L)$ would
also be a finitely generated $A$-algebra as can be easily checked.
But again it is probably easier to verify finite generation on the 
$L_i$ rather directly on $L$. 
\end{rem}

\begin{rem}
Let $X$ be a normal variety over an algebraically closed field $k$.
We can interpret sections of divisors on $X$ 
as rational functions. For example if $D$ is a Weil divisor on $X$, then we can 
describe $\x{O}_X(D)$ in a canonical way as 
$$
\x{O}_X(D)(U)=\{f\in K(X) \mid (f)+D|_U\ge 0\}
$$
where $U$ is an open subset of $X$ and $K(X)$ is the function field of $X$.
If $D'\sim D$, then of course $\x{O}_X(D')\simeq \x{O}_X(D)$ but these sheaves 
are not identical, i.e. they have different embedding in the constant sheaf associated 
to $K(X)$. To restrict the rational functions in $H^0(X,D)=\x{O}_X(D)(X)$ to a prime divisor  
$S$ we usually need to move $D$ to get it right. It is enough to choose $D'\sim D$ so that 
$S$ is not a component of $D'$. In this case, the rational functions in $H^0(X,D')$ 
have no poles along $S$ hence they can be restricted to $S$. Moreover, if $D''\sim D$ 
is another divisor whose support does not contain $S$, then the restrictions of 
$H^0(X,D')$ and $H^0(X,D'')$ to $S$ are isomorphic as $k$-vector spaces.
\end{rem}

\begin{exa}\label{exa-lcm-illustration}
In this example, we illustrate Theorem \ref{c-dm-2} in a more familiar setting. 
Let $X$ be a smooth projective variety over an algebraically closed field $k$.
Assume that $L=L_1+2L_2$ where $L_1,L_2$ are distinct prime divisors. 
Let $L_i'\sim L_i$ so that $L_1,L_2,L_1',L_2'$ are pairwise with no 
common components. Let $L'=L_1'+2L_2'$. 
Obviously, $R(L)$ is a finitely generated $k$-algebra iff $R(L')$ is so. 
Assume that  the restriction of 
$R(L')$ to $L_1$ is a finitely generated $k$-algebra. We have the exact sequence 
$$
0\to M^0_{-L_1}(L')\to R(L') \to R(L')|_{L_1} \to 0
$$
According to Lemma \ref{l-modules-to-algebras}, $R(L')$ is a finitely generated 
$k$-algebra iff $M^0_{-L_1}(L')$ is a finitely generated $R(L')$-module.

Now $M^0_{-L_1}(L')\simeq M^0_{-L_1'}(L')$ and we have the exact sequence 
$$
0\to M^0_{-L_1'-L_2}(L')
\to M^0_{-L_1'}(L') \to M^0_{-L_1'}(L')|_{L_2} \to 0
$$
So, assuming that $M^0_{-L_1'}(L')|_{L_2}$ is a finitely generated module over  
the restriction $R(L')|_{L_2}$, $M^0_{-L_1'}(L')$ is a finitely generated 
$R(L')$-module iff $M^0_{-L_1'-L_2}(L')$ is a finitely generated $R(L')$-module. 
Moreover, $M^0_{-L_1'-L_2}(L')\simeq M^0_{-L_1'-L_2'}(L')$ and we have the exact sequence 
$$
0\to M^0_{-L_1'-L_2'-L_2}(L') \to M^0_{-L_1'-L_2'}(L') \to M^0_{-L_1'-L_2'}(L')|_{L_2} \to 0
$$
So, assuming that $M^0_{-L_1'-L_2'}(L')|_{L_2}$ is finitely generated over $R(L')|_{L_2}$, 
 $M^0_{-L_1'-L_2'}(L')$ is a finitely generated 
$R(L')$-module iff $M^0_{-L_1'-L_2'-L_2}(L')$ is a finitely generated $R(L')$-module. 
Finally, note that  
$$
M^0_{-L_1'-L_2'-L_2}(L')\simeq M^0_{-L_1'-2L_2'}(L')=M^0_{-L'}(L')
$$ 
and that $M^0_{-L'}(L')$ is finitely generated as an $R(L')$-module 
because if $\beta \in M^0_{-L'}(L')$ is homogeneous of degree $m>0$, 
then $\beta=\beta \cdot 1$ where the $\beta$ on the right hand side 
is considered as an element of $R(L')$ of degree $m-1$ and the $1$ is 
considered as an element of $M^0_{-L'}(L')$ of degree $1$. So, we have 
the desired finite generation.\\
\end{exa}

\begin{rem}
Let $(X,B)$ be a projective log smooth dlt pair over $\C$, and $B$ rational. 
Assume that $I(K_X+B)$ is integral for some positive integer $I$. 
Moreover, assume that $I(K_X+B)\sim L=\sum l_iL_i\ge 0$ where $L_i$ 
are prime divisors. Let $S=L_j$ for some $j$. To prove that $R(L)|_S$ is a 
finitely generated algebra, one can hope to somehow relate it to the algebra 
$R(L|_S)$. This is difficult to achieve for arbitrary $S$. 
However, if $S$ is a component of $\rddown{B}$, then it is expected 
that finite generation of $R(L)|_S$ is closely linked with 
the finite generation of $R(L|_S)$ because $L|_S\sim I(K_S+B_S)=I(K_X+B)|_S$ 
where $(S,B_S)$ is dlt. The ideal situation is when every 
component of $L$ is a component of $\rddown{B}$ (i.e. $\theta(X,B,L)=0$ 
with the notation in [\ref{B-zd}, Definition 2.1]). 

In general, when $\theta(X,B,L)\neq 0$,  
one can increase $B$ to create $\overline{B}\ge B$ so that 
$$
\Supp(\overline{B}-B)\subseteq \Supp L\subseteq \Supp \rddown{\overline{B}}
$$
in particular 
$\theta(X,\overline{B},\overline{L})=0$ where $\overline{L}\sim I(K_X+\overline{B})$ 
and $\Supp \overline{L}=\Supp L$. One could first try to prove that $R(\overline{L})$ is finitely generated 
and then try to chace the finite generation back to $R(L)$. 
This process is one of the main techniques employed in the papers [\ref{BCHM}]
[\ref{B}][\ref{B-zd}][\ref{BP}] in the context of the minimal model program.
\end{rem}

\vspace{0.4cm}

\section{Log canonical modules and minimal models}
\vspace{0.1cm}

\textbf{Set up.} In this section, we work with quasi-projective varieties over $k=\C$
unless stated otherwise. Throughout the section, we let $X\to Z=\Spec A$ be a projective morphism of normal 
varieties over $k$ with $Z$ being \emph{affine}. In some places we take a boundary $B$ on $X$. 

We use the 
notion and notation of pairs and log minimal models as in [\ref{B-zd}].
We use the numerical Kodaira dimension $\kappa_\sigma$ as introduced by Nakayama [\ref{Nakayama}].
Let $D$ be a Weil divisor on $X$.  If  $h^0(X,D)\neq 0$, we let $\Fix D$ to be the largest 
effective Weil divisor satisfying $\Fix D\le D'$ for any effective divisor 
$D'\sim D$. We let the movable part of $D$ to be $\Mov D=D-\Fix D$. In particular, 
$H^0(X,\Mov D)=H^0(X,D)$.

\begin{thm}\label{t-lcm-abundance-to-fg}
Assume that $(X/Z,B)$ is lc and $L:=I(K_X+B)$ is Cartier for some integer $I>0$. 
Assume further that $(X/Z,B)$ has a log minimal model $(Y/Z,B_Y)$ on which $|I(K_Y+B_Y)|$ is base point 
free.  Then, $R(L)$ is a finitely generated $A$-algebra, and 
$M_{\x{F}}^p(L)$ is a finitely generated $R(L)$-module 
 for any $p$ and any torsion-free coherent sheaf $\x{F}$.
\end{thm}
\begin{proof}
Let $f\colon W\to X$ and $g\colon W\to Y$ be a common resolution. Then, we can write 
$f^*I(K_X+B)=g^*I(K_Y+B_Y)+E$ where $E\ge 0$ and exceptional$/Y$ [\ref{B-zd}, Remark 2.4].
 Then, by letting $L_Y:=I(K_Y+B_Y)$ we have $R(L)\simeq R(L_Y)$ as $A$-algebras 
and this is a finitely generated $A$-algebra as $|L_Y|$ is base point free by assumptions.
In particular, the above ring is Noetherian.
Let $\x{G}$ be any torsion-free coherent sheaf on $Y$ and let $\pi\colon Y\to T/Z$ be the 
contraction defined by $|L_Y|$. There is a very ample$/Z$ divisor $N$ on $T$ such 
that $L_Y\sim \pi^* N$. Then, by the projection formula 
$$
\pi_*(\x{G}(mL_Y))\simeq (\pi_*\x{G})(mN)
$$
hence 
$$
H^0(Y,\x{G}(mL_Y))\simeq H^0(T,(\pi_*\x{G})(mN))
$$
so $R(L_Y)\simeq R(N)$ as $A$-algebras and 
$
M_{\x{G}}^p(L_Y)\simeq M_{\pi_*\x{G}}^p(N)
$ 
as modules. By Theorem \ref{t-lcm-big-divisors}, $M_{\pi_*\x{G}}^p(N)$ is a finitely generated 
$R(N)$-module hence $M_{\x{G}}^p(L_Y)$ is a finitely generated $R(L_Y)$-module.

Now we prove the finite generation of $M_{\x{F}}^p(L)$. 
By Theorem \ref{t-lmc-ample-to-general} (1), 
we may assume that $\x{F}=\x{O}_X(G)$ where $G$ is some very ample$/Z$ divisor.
We have isomorphisms
$$
H^0(X,mL+G)\simeq H^0(W,f^*mL+f^*G) 
$$
and this is isomorphic to a subspace of 
$H^0(Y,mL_Y+g_*f^*G)$. So, $M_{G}^p(L)$ is isomorphic to a 
submodule of $M_{g_*f^*G}^p(L_Y)$. Therefore, $M_{G}^p(L)$ is a finitely 
generated $R(L)$-module as $M_{g_*f_*G}^p(L_Y)$ is a finitely 
generated $R(L_Y)$-module.\\
\end{proof}

\begin{thm}\label{t-lcm-fg-abundance}
Assume that $Z=\rm pt$ and  $L$ is a Cartier divisor on $X$. 
Assume that for any very ample divisor $G$ 
the module $M_{G}^0(L)$ is finitely generated over $R(L)$. Then, 
abundance holds for $L$, that is, $\kappa(L)=\kappa_\sigma(L)$.
\end{thm}
\begin{proof}
The inequality $\kappa(L)\le \kappa_\sigma(L)$ follows from the fact that 
$\kappa(L)=\kappa(JL)$ and $\kappa_\sigma(L)=\kappa_\sigma(JL)$ for any 
positive integer $J$ and the fact that for some $J$ 
and certain constants $c_1,c_2>0$ we have
$$
c_1m^{\kappa(L)}\le h^0(X,mJL) \le c_2m^{\kappa(L)}
$$
for any $m\gg 0$.

For the converse $\kappa(L)\ge \kappa_\sigma(L)$, we may assume that $\kappa_\sigma(L)\ge 0$ 
and we can choose a very ample divisor $G$ so that $\kappa_\sigma(L)$ satisfies 
$$
\limsup_{m\to +\infty} \frac{h^0(X,mL+G)}{m^{\kappa_\sigma(L)}}>0 
$$
By assumptions,  $M_{G}^0(L)$ is a finitely generated $R(L)$-module. 
Let $\{\alpha_1,\dots,\alpha_r\}$ be a set of generators of homogeneous elements 
with $n_i:=\deg \alpha_i$. For any $\alpha\in M^0_G(L)$ of degree $m$, there are 
homogeneous elements $a_i\in R(L)$ such that $\alpha=\sum_i a_i\alpha_i$. 
It is clear that $\deg a_i=m-n_i$. Thus, 
$$
 h^0(X,(m-n_1)L)+\dots+h^0(X,(m-n_r)L)\ge h^0(X,mL+G)
$$ 
which implies that 
$$
\limsup_{m\to +\infty} \frac{h^0(X,(m-n_1)L)+\dots+h^0(X,(m-n_r)L)}{m^{\kappa_\sigma(L)}}>0
$$
hence $\kappa(L)\ge \kappa_\sigma(L)$.\\
\end{proof}

The following theorem is well-known (cf. [\ref{pl-flips}]).

\begin{thm}\label{t-lcm-fg-sf} 
Let $L$ be a Cartier divisor on $X$ with $h^0(X,nL)\neq 0$ for some $n>0$. 
Then, the following are equivalent:

$(1)$ $R(L)$ is a finitely generated $A$-algebra;

$(2)$ there exist a projective 
birational morphism $f\colon W\to X$ from a smooth variety, a positive integer 
$J$, and Cartier divisors $E$ and $F$ such that $|F|$ is base point free, and 
$$
\Mov f^*mJL=mF ~~~\mbox{and}~~~ \Fix f^*mJL=mE
$$ 
 for every positive integer $m$. 
\end{thm}
\begin{proof}
Assume that $R(L)$ is a finitely generated $A$-algebra. Perhaps after replacing 
$L$ with $JL$ for some positive integer 
$J$, we may assume that the algebra $R(L)$ is generated by elements 
$\alpha_1,\dots, \alpha_r$ of degree $1$, 
and that there is a resolution $f\colon W\to X$ on which $f^*L=F+E$ where $F$ is free, 
$\Mov f^*L=F$, and $\Fix f^*L=E$. We could in addition assume that $F\ge 0$ with no 
common components with $E$.
 Obviously, $\Fix mf^*L\le mE$ for any $m>0$. 
Suppose that equality does not hold for some $m>0$. Take $m>0$ minimal with this property. 
There is $\alpha\in H^0(W,mf^*L)$ and a component 
$S$ of $E$ such that $\mu_S(\alpha)<0$ where $\mu$ stands for multiplicity, 
that is, the coefficient and $(\alpha)$ is the divisor associated to the 
rational function $\alpha$.  
Since $E=\Fix f^*L$, $m>1$. By assumptions, 
$\alpha=\sum a_i\alpha_i$ where $a_i$ are elements 
of $H^0(W,(m-1)f^*L)$. Thus, 
$$
\mu_S(\alpha)\ge \min\{\mu_S(a_i)+\mu_S(\alpha_i)\}
$$
Assume that the minimum is equal to $\mu_S(a_j)+\mu_S(\alpha_j)$ for some $j$. 
Since $E=\Fix f^*L$, we have $\mu_S(\alpha_j)\ge 0$ hence  
$0>\mu_S(a_j)+\mu_S(\alpha_j)\ge \mu_S(a_j)$. This is a  contradiction since 
the minimality of $m$ ensures that $\mu_S(a_j)\ge 0$.

Conversely, assume that there exist $f\colon W\to X$,  
$J$, $E$, and $F$ as in the theorem. Then, $R(JL)\simeq R(f^*JL)\simeq R(F)$ 
is a finitely generated $A$-algebra as $|F|$ is base point free. This implies 
that $R(L)$ is a finitely generated $A$-algebra by the so-called truncation principle [\ref{pl-flips}, Theorem 4.6].\\ 
\end{proof}

\begin{thm}\label{l-lcm-strong-sf} 
Let $L$ be a Cartier divisor on $X$ with $h^0(X,nL)\neq 0$ for some $n>0$ and that 
$R(L)$ is a finitely generated $A$-algebra.
Assume further that $M^0_{G'}(L)$ is a finitely generated 
$R(L)$-module for any very ample$/Z$ divisor $G'$. Let $f,W,F,E,J$ be as in Theorem \ref{t-lcm-fg-sf}. Fix a 
nonnegative integer $r$ and a very ample$/Z$ divisor $G$ on $W$. Then, 
$$
\Supp \Fix(m(f^*JL+rF)+G)=\Supp E
$$
for every integer $m\gg 0$.
\end{thm}
\begin{proof}
Let $G'$ be a very ample$/Z$ divisor on $X$ such that $G\le f^*G'$. 
 By assumptions, $R(L)$ is a Noetherian ring and $M=M^0_{G'}(L)$ is a Noetherian $R(L)$-module. Moreover, 
$R(L)$ is integral over the ring $R(L)^{[J]}$ which implies that $M$ is a finitely generated 
$R(L)^{[J]}$-module. Put 
$N_0=\bigoplus_{m\ge 0} N_{m,0}$ where $N_{m,0}=M_m$ if $J|m$ but 
$N_{m,0}=0$ otherwise, as in Remark \ref{r-lcm-truncation}. Since $N_0$ is an $R(L)^{[J]}$-submodule, 
it is finitely generated over $R(L)^{[J]}$.
This corresponds to saying that $M_{G'}^0(JL)$ is a finitely generated $R(JL)$-module.  
On the other hand, $M_{G}^0(f^*JL)$ is a submodule of  $M_{f^*G'}^0(f^*JL)$ hence a 
finitely generated $R(f^*JL)$-module.
Thus, after replacing $L$ with $f^*JL$ and $X$ with $W$ we can assume that $J=1$ and $W=X$.
We may also assume that $F,G\ge 0$ and that $F+G$ has no common component with $E$.

Obviously, 
$$
\Supp \Fix(m(L+rF)+G) \subseteq\Supp E
$$
 for every integer $m>0$.
 Assume that there is a component 
$S$ of $E$ which does not belong to $\Supp \Fix(m(L+rF)+G)$ for some $m>0$. 
Let $\alpha\in H^0(X,m(L+rF)+G)$ so that 
$$
S\nsubseteq \Supp ((\alpha)+m(L+rF)+G)
$$
which in particular means that $\mu_S(\alpha)=-m\mu_SE$.
Since 
$$
(m+mr)L+G=m(L+rF)+G+mrE
$$
and $mE\ge 0$, there is $\alpha'$ in  $M_{G}^0(L)$ of degree $m+mr$ such that $\alpha'=\alpha$ 
as rational functions on $X$.

Assume that $\{\alpha_1,\dots,\alpha_r\}$ is a set of homogeneous generators of $M_{G}^0(L)$ 
with $n_i:=\deg \alpha_i$.
We can write $\alpha'=\sum a_i\alpha_i$ where $a_i\in R(L)$ is homogenous of 
degree $m+mr-n_i$. Therefore, 
$$
\mu_S(\alpha')\ge \min\{\mu_S(a_i)+\mu_S(\alpha_i)\}
$$
Since 
$$
\Fix (m+mr-n_i)L=(m+mr-n_i)E
$$ 
we have $\mu_S(a_i)\ge 0$ hence if the above minimum is attained 
at index $j$, then 
$$
-m\mu_SE=\mu_S(\alpha)=\mu_S(\alpha')\ge \mu_S(\alpha_j)
$$ 
from which we get $m\mu_SE\le -\mu_S(\alpha_j)$. This means that such $m$ cannot be too large so 
the theorem holds for $m\gg 0$.\\
\end{proof}

Next, using the results of [\ref{BCHM}], we can prove the converse of 
Theorem \ref{t-lcm-abundance-to-fg}. 

\begin{thm}\label{t-lcm-abundance-from-fg}
Assume that $(X/Z,B)$ is lc and assume that there is a positive 
integer $I$ such that

$(1)$ $L:=I(K_X+B)$ is Cartier and pseudo-effective$/Z$, 

$(2)$ $R(L)$ is a finitely generated $A$-algebra, and 

$(3)$ for any very ample$/Z$ divisor $G$ 
the module $M_{G}^0(L)$ is finitely generated over $R(L)$. 

Then, there is a log minimal model $(Y/Z,B_Y)$ for $(X/Z,B)$ on which $K_Y+B_Y$ is semi-ample$/Z$.
\end{thm}
\begin{proof}
We may assume that $X\to Z$ is surjective. Let $V$ be the generic fibre of 
$X\to Z$. As $Z$ is affine, by base change theorems, $R(L|_V)\simeq R(L)\otimes_AK$ is a finitely generated $K$-algebra, and 
$M^0_{G|_V}(L|_V)\simeq M^0_G(L)\otimes_AK$ is a finitely generated $R(L|_V)$-module where $K$ is the 
function field of $Z$ and $G$ is any very ample$/Z$ divisor on $X$. By Theorem \ref{t-lmc-ample-to-general}
and Theorem \ref{t-lcm-fg-abundance}, $\kappa(L|_V)\ge 0$ which in particular implies that 
$h^0(X,nL)\neq 0$ for some $n>0$.

 Let $f,W,E,F,J$ be as in Theorem \ref{t-lcm-fg-sf}. We may assume that $f$ gives a log resolution of 
 $(X/Z,B)$. Let $B_W$ be $B^\sim$ plus the reduced 
 exceptional divisor of $f$ where $B^\sim$ is the birational transform of $B$.
 We can write 
$$
JI(K_W+B_W)=JIf^*(K_X+B)+E'
$$
 where $E'\ge 0$ is exceptional$/X$. Pick any $\alpha\in H^0(W,mJI(K_W+B_W))$ and 
let 
$$
P:=(\alpha)+mJI(K_W+B_W)= (\alpha)+mF+mE+mE'
$$ 
Since $P-mE'\equiv 0/X$, $P-mE'\ge 0$ by the negativity lemma. 
On the other hand, $\Fix (P-mE')=mE$ hence $P-mE'\ge mE$. This implies that 
$$
\Fix mJI(K_W+B_W)=mE+mE' ~~~\mbox{and} ~~~\Mov mJI(K_W+B_W)=mF
$$
It is enough to construct a log minimal model for 
$(W/Z,B_W)$ with the required properties. Note that $(W/Z,B_W)$ satisfies the conditions (1) and (2) 
of the theorem, automatically. Condition (3) is also preserved: if $G$ is a very ample$/Z$ divisor 
on $W$, there is a very ample$/Z$ divisor $G'$ on $X$ such that $G\le f^*G'$ hence 
$M^0_{G}(I(K_W+B_W))$ is a finitely generated $R(I(K_W+B_W))$-module as it is a submodule 
of $M^0_{f^*G'}(I(K_W+B_W))\simeq M^0_{G'}(I(K_X+B))$.
Therefore, 
by replacing $(X/Z,B)$ with $(W/Z,B_W)$ from now on we can assume that 
$W=X$, $f$ is the identity and $E'=0$. 
Let  
 $g\colon X\to T$ be the contraction$/Z$ defined by $|F|$. 

Let  $F'$ be a general element of $|F|$. 
We can choose a very ample$/Z$ divisor $G\ge 0$ 
so that $K_X+B+rF'+G$ 
is nef$/Z$ and that $(X/Z,B+rF'+G)$ is dlt. 
Run the LMMP$/Z$ on $K_X+B+rF'$ with scaling of $G$.  
By boundedness of the length of extremal rays due to Kawamata, 
if $r$ is a sufficietly large integer, then the LMMP 
is over $T$, i.e. only extremal rays over $T$ are contracted. 
Suppose that, perhaps after some log flips and divisorial contractions, we get an infinite sequence of 
log flips $X_i\bir X_i/Z_i$. Let $\lambda_i$ be the numbers appearing in the 
LMMP with scaling in the above sequence of log flips, that is, $K_{X_i}+B_i+rF_{i}'+\lambda_iG_i$ 
is nef$/Z$ and numerically trivial over $Z_i$ where $B_i,F_i',G_i$ are birational 
transforms on $X_i$.  By [\ref{BCHM}], $\lambda:=\lim \lambda_i=0$.
Moreover, for each $i$, $K_{X_i}+B_{i}+rF_{i}'+\lambda_iG_{i}$ is semi-ample$/Z$. Thus, 
if $S$ is a component of $E$ not contracted by the LMMP, then there exist 
$$
0\le N_{i}\sim_\Q K_X+B+rF'+\lambda_i G
$$
not containing $S$.
This contradicts Theorem \ref{c-n-sigma} and Theorem \ref{l-lcm-strong-sf}. 
Therefore, $E$ is contracted by the LMMP and $K_{X_i}+B_{i}+rF_{i}'$ is 
$\Q$-linearly a multiple of $F_{i}'$. 
But $|F_{i}'|$ is base point free as the LMMP we ran is over $T$. Thus, 
the LMMP terminates with a log minimal model satisfying the desired properties.\\ 
\end{proof}

The following theorem was proved by Nakayama [\ref{Nakayama}, Theorem 6.1.3]. 
He treated the case $Z=\rm pt$ but his proof works for general $Z$. For convenience of 
the reader we present his proof.

\begin{thm}\label{c-n-sigma}
Assume that $W\to Z=\Spec A$ is a projective morphism from a smooth variety, 
$w\in W$ a closed point,  and $D$ a Cartier divisor on 
$W$. Assume further that for some effective divisor $C$ there exist an infinite 
sequence of positive rational numbers $t_1>t_2>\cdots$ with $\lim t_i=0$, 
and effective $\Q$-divisors 
$N_{i}\sim_\Q D+t_iC$ with $w\notin \Supp N_{i}$. Then, there is a very ample$/Z$ divisor 
$G$ on $W$ such that $w\notin \Bs|mD+G|$ for any $m>0$.  
\end{thm}
\begin{proof}
Let $f\colon W'\to W$ be the blow up at $w$ with $E$ the exceptional divisor, $D'=f^*D$, $C'=f^*C$, and 
$N_{i}'=f^*N_{i}$. Let $G$ be a very ample$/Z$ divisor on $W$ such that 
$H:=f^*G-K_{W'}-E$ is ample$/Z$ and $H-\epsilon C'$ is also ample$/Z$ for some 
$\epsilon>0$. Put $G'=f^*G$. For each $m> 0$, we can write 
$$
mD'+G'=K_{W'}+E+H+mD'=K_{W'}+E+H-mt_iC'+m(t_iC'+D')
$$
$$
\sim_\Q K_{W'}+E+H-mt_iC'+mN_{i}'
$$
where  we choose $t_i$ so that $mt_i<\epsilon$.
By assumptions, $E$ does not intersect $N_{i}'$. Thus, the multiplier ideal sheaf 
$\x{I}_i$ of $mN_{i}'$ is isomorphic to $\x{O}_{W'}$ near $E$. In particular, we have 
the natural exact sequence 
$$
0\to \x{I}_i(mD'+G'-E) \to \x{I}_i(mD'+G')\to \x{O}_{E}(mD'+G')\to 0
$$
from which we derive the exact sequence 
$$
H^0(W',\x{I}_i(mD'+G'))\to H^0(E,(mD'+G')|_E) 
\to H^1(W',\x{I}_i(mD'+G'-E))=0
$$
where the last vanishing follows from Nadel vanishing. On the other hand, 
$(mD'+G')|_E\sim 0$ hence some section of  $\x{I}_i(mD'+G')$ does not 
vanish on $E$. But 
$$
\x{I}_i(mD'+G')\subseteq \x{O}_{W'}(mD'+G')
$$ 
so some section of $\x{O}_{W'}(mD'+G')$  does not vanish on $E$ which simply 
means that $w$ is not in $\Bs|mD+G|$.\\
\end{proof}

The next result is useful for inductive treatment of finite generation as 
illustrated in [\ref{B-lc-flips}].

\begin{thm}\label{t-fg-pullbacks}
Let $f\colon X\to Y/Z$ be a surjective morphism of normal varieties, projective 
over an affine variety $Z=\Spec R$, and $L$ a Cartier divisor on $Y$. 
If $R(X/Z,f^*L)$ is a finitely generated $R$-algebra,  
then $R(Y/Z,L)$ is also a finitely generated $R$-algebra.
\end{thm}
\begin{proof}
If $f$ is a contraction, the claim is trivial by the projection formula.
So, replacing $f$ by the finite part of the Stein factorisation we may 
assume that $f$ is finite. We have a natural injective $\x{O}_Y$-morphism 
$\phi\colon \x{O}_Y\to f_*\x{O}_X$. Moreover, since we work over $\C$, we also have a 
$\x{O}_Y$-morphism $\psi\colon f_*\x{O}_X\to \x{O}_Y$ given by $\frac{1}{\deg f}{\rm Trace}_{X/Y}$ 
such that $\psi\phi$ is the identity morphism on $\x{O}_Y$ (cf. [\ref{Kollar-Mori}, Proposition 5.7]).
We actually have a splitting $f_*\x{O}_X\simeq \x{O}_Y\oplus \x{F}$ 
such that $\phi$ corresponds to the natural injection and $\psi$ to the first projection. 
So, for each $m>0$, we have induced maps 
$$
H^0(Y,\x{O}_Y(mL)) \to H^0(Y,(f_*\x{O}_X)(mL))\simeq  H^0(X,\x{O}_X(mf^*L))
$$
$$
H^0(Y,(f_*\x{O}_X)(mL))\simeq  H^0(X,(\x{O}_X(mf^*L))\to H^0(Y,\x{O}_Y(mL))
$$
where the upper map is injective and the lower map is surjective.
Thus, we get an injection $\pi\colon R(Y/Z,L)\to R(X/Z,f^*L)$ and a surjection 
$\mu\colon R(X/Z,f^*L)\to R(Y/Z,L)$ 
whose composition gives the identity map on $R(Y/Z,L)$. Here, $\pi$ is an $R(Y/Z,L)$-algebra 
homomorphism but $\mu$ is only an $R(Y/Z,L)$-module homomorphism. 
Let $I$ be an ideal of $R(Y/Z,L)$ and let $I'$ be the ideal of $R(X/Z,f^*L)$ generated 
by $\pi(I)$. We show that $\pi^{-1}I'=I$. Let $c\in \pi^{-1}I'$. We can write 
$\pi(c)=\sum \pi(a_i)\alpha_i$ where $a_i\in I$ and $\alpha_i\in R(X/Z,f^*L)$. 
Now, 
$$
c=\mu(\pi(c))=\mu(\sum \pi(a_i)\alpha_i)=\sum \mu(\pi(a_i))\mu(\alpha_i)=\sum a_i\mu(\alpha_i)
$$
hence $c$ is in the ideal $I$. 

 Assume that $R(X/Z,f^*L)$ is a finitely generated $R$-algebra. Then, $R(X/Z,f^*L)$ 
 is Noetherian.
Let $I_1\subseteq I_2\subseteq \cdots$ be a chain of ideals of $R(Y/Z,L)$ and let 
$I_1'\subseteq I_2'\subseteq \cdots$ be the corresponding chain of ideals in $R(X/Z,f^*L)$. 
Since $R(X/Z,f^*L)$ is Noetherian, the latter chain stabilises which implies that 
the former sequence also stabilises as $\pi^{-1}I_j'=I_j$. Therefore, $R(Y/Z,L)$ is Noetherian which is 
equivalent to saying that it is a finitely generated $R$-algebra because $R(Y/Z,L)$ 
is a graded $R$-algebra and the zero-degree piece of $R(Y/Z,L)$ is a finitely generated 
$R$-algebra.\\
\end{proof}



\vspace{2cm}

\flushleft{DPMMS}, Centre for Mathematical Sciences,\\
Cambridge University,\\
Wilberforce Road,\\
Cambridge, CB3 0WB,\\
UK\\
email: c.birkar@dpmms.cam.ac.uk

\vspace{0.5cm}

Fondation Sciences Math\'ematiques de Paris,\\
IHP, 11 rue Pierre et Marie Curie,\\
75005 Paris,\\
France


\begin{thebibliography}{99}


\bibitem{}\label{B-lc-flips}  {C. Birkar; {\emph{Existence of log canonical flips and a special LMMP.}} 	arXiv:1104.4981v1. }

\bibitem{}\label{B-zd}  {C. Birkar; {\emph{On existence of log minimal models and weak Zariski decompositions.}} 
arXiv:0907.5182v1. To appear in Math. Annalen. }


\bibitem{}\label{B-topics}  {C. Birkar; {\emph{Topics in algebraic geometry.}} 
Lecture notes of a graduate course.	arXiv:1104.5035v1 }

\bibitem{}\label{B}  {C. Birkar; {\emph{On existence of log minimal models.}}  {Compositio Math.} volume 145 (2009), 1442-1446. }

\bibitem{}\label{BCHM}  {C. Birkar, P. Cascini, C. Hacon, J. M$^c$Kernan; {\emph{Existence of minimal models
for varieties of log general type.}}  J. Amer. Math. Soc. 23 (2010), 405-468. }

\bibitem{}\label{BP}  {C. Birkar, M. P\u{a}un; {\emph{Minimal models, flips and finite generation : a tribute to 
V.V. SHOKUROV and Y.-T. SIU.}}  In "Classification of algebraic varieties", 
European Math Society series of congress reports (2010).}


\bibitem{}\label{EGA} J. Dieudonn\'e, A. Grothendieck; \emph{\'El\'ements de g\'eom\'etrie alg\'ebrique, EGA III-1}. Publications Math\'ematiques de l'IH\'ES. Available online.

\bibitem{}\label{Hartshorne} R. Hartshorne: \emph{Generalized divisors on Gorenstein schemes.} 
K-theory 8 (1994), 287-339.

\bibitem{}\label{Kollar-Mori}  {J. Koll\'ar, S. Mori; {\emph{Birational Geometry of Algebraic Varieties.}} Cambridge University
Press (1998).}

\bibitem{}\label{Nakayama} {N. Nakayama; {\emph{Zariski decomposition and abundance}}. 
MSJ Memoirs  {\bf 14}, Tokyo (2004).}

\bibitem{}\label{Schwede} K. Schwede; Generalized divisors and reflexive sheaves.

\bibitem{}\label{pl-flips} V.V. Shokurov; \emph{Prelimiting flips}. 
Proc. Steklov Inst. Math. 240 (2003), 75-213.

\end{thebibliography}
\end{document}